\documentclass[11pt]{amsart}
\usepackage{geometry}                
\usepackage{graphicx}
\usepackage{amssymb}
\usepackage{epstopdf}
\usepackage{amsmath}
\usepackage{mathtools}
\usepackage{txfonts}
\usepackage{url}
\usepackage{parskip}

\newtheorem{theorem}{Theorem}[section]
\newtheorem{lemma}[theorem]{Lemma}
\newtheorem{proposition}[theorem]{Proposition}
\newtheorem{corollary}[theorem]{Corollary}
\newtheorem{conjecture}[theorem]{Conjecture}

\newtheorem{definition}[theorem]{Definition}

\title{Fermat's Last Theorem: Algebra, Geometry, and Number Theory}
\author{Felix Sidokhine}
\date{\today}                                           
\begin{document}
\maketitle

\begin{abstract}In our work we give the examples using Fermat's Last Theorem for solving some problems from algebra, geometry and number theory
\end{abstract}

\section{The Fermat Theorem and Diophantine Equations of the Second Degree}

\subsection{The equation $X^2 + Y^2 = Z^2$, Pythagorean triples}

\begin{proposition}
Let a polynomial $p(x)=x^2 + a^nx -b^n$, where $a,b$ $(ab \neq 0, b>0)$ are co-prime integers and $n$ is a positive integer be given. The polynomial $p(x)$ is reducible over $\mathbb{Q}$ if and only if the equation $X^{2n} + 4Y^n = Z^2$, where $\gcd(X,2Y)=1$ has a solution over $\mathbb{Z}^+$.
\end{proposition} 

\begin{proof}
Let $p(x)$ be reducible over $\mathbb{Q}$ (without loss of generality assume that $a \nequiv 0 \mod 2$), then the discriminant of $p(x)$ is a faithful square. Therefore, $a^{2n} + 4b^n = c^2$ for some $c \in \mathbb{Z}$. This is the same as looking at $X^{2n} + 4Y^n = Z^2$ where $\gcd(X,2Y)=1$. Suppose $X^{2n} + 4Y^n = Z^2$ has a solution over $\mathbb{Z}^+$, then there is a polynomial $p(x)=x^2 + a^nx - b^n$ reducible over $\mathbb{Q}$.
\end{proof}

\begin{theorem}
The equation $X^2 + Y^2 = Z^2$ has an integral solution of the form $(X = zx^n, Y=2y^n, Z=x^{2n} + 2y^n)$, where $x,2y,z$ are pairwise co-prime for $n \leq 2$ and no solutions if $n > 2$.
\end{theorem}

\begin{proof}

The equation $X^{2n} + 4Y^n = Z^2$ can be re-written as $(ZX^n)^2 + (2Y^n)^2 = (X^{2n} + 2Y^n)^2$. Consider the following three cases:

\begin{itemize}

	\item Let $n=1$, then $(ZX)^2 + (2Y)^2 = (X^2 + 2Y)^2$ and its solutions $(X,Y,Z)$ satisfy the parametric relations $X=u-v, Y=uv, Z = u+v$ where $\gcd(u,v)=1$ $uv \equiv 0 \mod 2$.
	
	\item Let $n=2$, then $(ZX^2)^2 + (2Y^2)^2 = (X^4 + 2Y^2)^2$ and its solutions $(X,Y,Z)$ satisfy the  parametric relations $X = (u^2 - v^2)^{\frac{1}{2}}, Y=uv, Z=u^2 + v^2$ where $\gcd(u,v)=1$ $uv \equiv 0 \mod 2$.
	
	\item Let $n > 2$, then $(ZX^n)^2 + (2Y^n)^2 = (X^{2n} + 2Y^n)^2$ and its solutions $(X,Y,Z)$ satisfy the  parametric relations $X = (u^n - v^n)^{\frac{1}{n}}, Y=uv, Z=u^n + v^n$, where $\gcd(u,v)=1$ $uv \equiv 0 \mod 2$.

\end{itemize}

Since Fermat's last theorem is true then the equation $X^2 + Y^2 = Z^2$ in integral variables where $\gcd(ZX,2Y)=1$ has no solutions of the form $(Z=zx^n,Y=2y^n,Z=x^n+2y^n)$, for $n \geq 3$.
\end{proof}

The equation $(ZX^n)^2 + (2kY^m)^2 = (X^{2n} \pm 2kY^m)^2$, a modification on $(ZX^n)^2 + (2Y^n)^2 = (X^{2n} + 2Y^n)^2$, and can be used to cover other cases, for example: 

$(ZX^2)^2 + (12Y^2)^2 = (X^4 - 12Y^2)^2$ over $\mathbb{Z}$, where $\gcd(ZX,6Y)=1$, has no solutions of the form $X=(2u^2 + 3v^2)^{\frac{1}{2}}, Y=uv, Z=2u^2 - 3v^2$ \cite{Davenport:1989aa}. However, this equation has infinitely many solutions of the form $X = (u^2 + 6v^2)^{\frac{1}{2}},Y=uv,Z=u^2 - 6v^2$.

The equation $(ZX^2)^2 + (2Y^m)^2 = (X^4 - 2Y^m)^2$ over $\mathbb{Z}$ where $\gcd(ZX,2Y)=1$ has infinitely many solutions of the form $X = (u^m + v^m)^{\frac{1}{2}}, Y=uv, Z=u^m -v^m$, when $m = 3$ \cite{Mordell:1969aa} and no such solutions for the case $m \geq 4$ \cite{Darmon:aa}.

\subsection{The equation $aX^2 + bY^2 = cZ^2$, Legendre's family of equations}

Fermat's theorem for $n=3$ was proven by Euler in 1753. He noticed that the proof seemed very different form the case when $n=4$, without ever having established the cause. We will try to answer this question.

Let us consider a family of Legendre's equations of the form $ax^2 + by^2 = cz^2$ where $a,b,c \in \mathbb{Z^+}$,  pairwise relatively prime, square-free and satisfy the following: $x^2 \equiv bc \mod a$, $x^2 \equiv ac \mod b$ and $x^2 \equiv -ab \mod c$, Then by Legendre's theorem such an equation (assuming $a,b$ and $c$ are fixed), has an integral solution and the number of solutions is infinite.

\begin{proposition}
Let $x^n + y^n = z^n$, $\gcd(x,y)=1$ and $n \geq 3$ have a solution, then there exists an equation $ax^2 + by^2 = cz^2$ were $a,b,c$ are relatively prime, square-free whose solution could be reduced to a solution of Fermat's equation.
\end{proposition}

\begin{proof}
Let the equation $x^n +y^n = z^n$ where $x,y,z$ are relatively prime and $n \geq 3$ is an odd integer have a solution $(x=a,y=b,z=c)$ where $a,b,c$ are pairwise relatively prime. Using the fundamental theorem of arithmetic we can can represent $a = \alpha d_1^2, b = \beta d_2^2, c=\gamma d_3^2$, where $\alpha,\beta,\gamma$ are pairwise relatively prime and square-free natural numbers. The equation $\alpha x^2 + \beta y^2 = \gamma z^2$ with $x = \alpha^kd_1^n,y=\beta^kd_2^n,z=\gamma^kd_3^n$ where $k = \frac{n-1}{2}$ is the equation with the desired solution.
\end{proof}

\begin{theorem}
No integral solution of a quadratic equation belonging to the family of Legendre's equations can be reduced to a solution of the equation $x^n + y^n = z^n$ when $n > 2$.
\end{theorem}

\begin{proof}
Theorem 1.4 is a direct consequence of proposition 1.3 and Fermat's theorem.
\end{proof}

As we can see for Fermat's theorem for $n=4$, the family of Legendre's equations consists of a unique quadratic equation then for any odd prime this family contains an infinite number of different Legendre's equations including the quadratic one.

We can give the following interpretation to Fermat's theorem: ``An equation $ax^2 + by^2 = cz^2$ of the family of Legendre's equations has no solutions of the form $x=a^kd_1^n,y=b^kd_2^n,z=c^kd_3^n$, where $n \geq 3$ is an odd integer and $k = \frac{n-1}{2}$. The equation $x^2 + y^2 = z^2$ has no solutions of the form $x = d_1^n,y=d_2^n,z=d_3^n, n \geq 2$".

The statement ``The equation $x^2 + y^2 = z^2$ has no solutions of the form $x=d_1^n, y=d_2^n, z=d_3^n$, $n \geq 2$" is a direct consequence of Abel's hypothesis (1823): ``Let $x,y,z$ be non-zero relatively prime integers such that $x^n + y^n = z^n$ $(n > 2)$ and $0 < x < y < z$ then none of the $x,y,z$ are prime powers'' \cite{Abel:1839aa}

A generalization of Abel's hypothesis could also be formulated: ``Let $ax^2 + by^2 = cz^2$ belong to a family of Legendre's equations and co-prime $x,y,z$ $(xyz \neq 1)$ are solutions then none of the $x,y,z$ are prime powers''.

An example from history is the equation $x^7 + y^7 = z^7$. The family of Legendre's equations have the following form $\alpha x^2 + \beta y^2 = \gamma z^2$ where $x = \alpha^3d_1^7, y= \beta^3 d_2^7, z=\gamma^3 d_3^7$. The case of the quadratic equation, $\alpha\beta\gamma = 1$ was studied by Dirichlet (1832) and the case including Legendre's equations $\alpha\beta\gamma \neq 1$ is square-free was studied by Lam\'{e} (1839).

Before Fermat's last theorem was proven by Wiles-Taylor, Fermat's problem got an unsuspected representation: ``Let Fermat's theorem be false and $a^n + b^n = c^n$ then an elliptic curve of the form $\mathbb{E}_{a,b,c}: y^2 = x(x-a^n)(x+b^n)$ is not modular".

\section{The Fermat Theorem and its Geometrical Aspects}

\subsection{Euclidean triangles}

Fermat's theorem over $\mathbb{Z}$ can be given the following geometric interpretation: ``The equation $x^n+y^n=z^n$ has a solution $x, y, z, n \in \mathbb{Z}^+$ if and only if the segments of length $x, y, z$ are the legs of a right-angle triangle or the length of one of the segments is equal to the sum of the lengths of the two others.''

Let us consider Fermat's theorem over the ring $\mathbb{Z}[\sqrt{2}]$. The equation $x^n +y^n = z^n$ has a solution in $\mathbb{Z}[\sqrt{2}]$ for $n = 2,3$. Thus, there exist acute and right-angle triangles with sides satisfying the equation $x^n +y^n = z^n$ for $n=2,3$. However, no such triangles exist for $n \geq 4$ \cite{Jarvis:2004aa}. If such solutions existed, then they must have been the sides of an acute triangle.

\subsection{$\mathbb{Z}$-modules and linear independence}

Let $\mathbb{Z}^3=\{(x, y, z) | \text{ where } x, y, z \in \mathbb{Z}\}$ be $\mathbb{Z}$-module.

\begin{definition}
$S = \{ (a, b, c) \in \mathbb{Z}^3 | a<b <c \text{ and } \gcd(a,b)=\gcd(ab,c)=1\}$
\end{definition}

\begin{theorem}
Let a submodule $L=\{(x, y, z)\in\mathbb{Z}^3 | ax+by-cz=0, (a, b, c)\in S\}$  and its $l_0,l_1$ linearly independent elements be given; then $l_0,l_1$ and $(a^k, b^k, c^k)$ are linearly independent over the field $\mathbb{Q}$ for $k \geq 2$, i.e. $l_0 \land l_1\land(a^k, b^k, c^k) \neq 0$.
\end{theorem}

\begin{proof}
Theorem 2.2 is direct consequence of Fermat's last theorem.
\end{proof}

\begin{corollary}
Let $L=\{(x, y, z)\in\mathbb{Z}^3 | ax+by-cz=0, (a, b, c)\in S\}$   and two linear independent elements $m_0, m_1$ of the module $\mathbb{Z}^3$  be given. If there is such $k\geq2$ that
$m_0\land m_1\land (a^k, b^k, c^k)=0$  is true then at least one of the elements $m_0, m_1 \notin L$.
\end{corollary}

The Euler conjecture as an extension of Fermat's theorem: ``The equation  $x^n+y^n+z^n=u^n$, where $x, y, z, u$  are pairwise relatively prime, has no solution over $\mathbb{Z}^+$ when $n>3$'',\cite{Sidokhine:aa} , can be also represented as a problem on the linear independence of the elements of $\mathbb{Z}$-module, namely:

The Euler Conjecture: Let $l_0,l_1$  and $l_2$ be three linear independent elements of an arbitrary submodule $L = \{(x, y, z, t)\in \mathbb{Z}^4  | ax + by + cz - dt = 0\}$, where $a, b, c$ and $d$ are the positive integers and relatively prime in pairs. Then the elements $l_0, l_1, l_2$  and $m_k= (a^k, b^k, c^k, d^k)$, where $k\geq 3$  is any integer, are linear independent over $\mathbb{Q}$.

\section{Fermat's Theorem and Diophantine Transcendental Equations}

\begin{definition}
$F=\{a, b, c \in \mathbb{Z}^+ | a<b<c<(a^2+ b^2 )^\frac{1}{2}\}$
\end{definition}

\begin{lemma} 
Let $(a, b, c) \in F$, then the equation $a^x+b^x-c^x=0$ with $x>2$ always has a solution over the field of real numbers $\mathbb{R}$ and this solution is unique.
\end{lemma}

\begin{lemma}
Let $a, b, c$ be the positive integers and satisfy $(a^2+ b^2 )^\frac{1}{2}<c$ then the equation $a^x+b^x-c^x=0$ with $x>2$  has no any solutions over the field of real numbers $\mathbb{R}$.
\end{lemma}

\begin{proposition}
The equation $x^\alpha+y^\alpha=z^\alpha$  in variables $x, y, z, \alpha$ where $x, y, z$ belong to positive integers and $\alpha > 2$ belongs to $\mathbb{R}$, has a solution if and only if $(x, y, z) \in F$. 
\end{proposition}

What is an arithmetical nature $\alpha$? If there is a set of the positive integers $(x, y, z)$ with $x<y<z<(x^2+ y^2 )^\frac{1}{2}$ such that the equalities: $x^\alpha+y^\alpha=z^\alpha$, where $\alpha$ is some real algebraic integer, or $x^e+y^e=z^e$, where $e$ is Napier's constant, or $x^\pi+y^\pi=z^\pi$, where $\pi$ is a constant, the ratio of a circle's circumference to its diameter, could be true?

\begin{conjecture}
For any real algebraic integer $n > 2$, the equation  $x^n+y^n=z^n$ has no solution in positive integers (an extension of Fermat's last theorem).
\end{conjecture}

The conjecture is a statement on Diophantine transcendental equation. Here is the other example: 

\begin{proposition}
The equation  $x^{2+ip}+y^{2+iq}=z^{2+ir}$ in rational integer variables $(x, y, z)$, where $xyz \neq 0$; $x, y, z$ are relatively prime, where $2 + ip, 2 + iq, 2 + ir$ are Gaussian integers and $p, q, r$  $(pq \neq 0)$ are integers, has no solution in the positive rational integers. 
\end{proposition}

The proof of the proposition is based on two statements:

\begin{enumerate}

\item $\Phi(a^u+b^v)=a^{\Phi(u)} + b^{\Phi(v)}$ , where $a, b$ are positive integers, $u, v$ are Gaussian integers and $\Phi$ is a conjugate operator of the complex numbers.

\item Theorem (Gelfond - Schneider). Let $\alpha$ be an algebraic number not equal to 0 or 1, and $\beta$ an algebraic number that is not rational. Then  $\alpha^\beta$  is a transcendental number, \cite{Baker:1979aa}.

\end{enumerate}

In the first such an approach for Gaussian integer exponents was offered and realized by Zuehlke, in 1999 \cite{Zuehlke:1999aa} . We would like to note that his approach is not limited only to $\mathbb{Z}[i]$, here is an example: 

\begin{proposition}
Let the system of equations in rational integer variables $x, y, z$ with algebraic integer exponents $u, v, w$ belonging to the ring $\mathbb{Z}[\sqrt{2}]$ have the form:

\begin{eqnarray*} 
ax^u  - by^v=cz^w \\
ax^{\Phi(u)}+by^{\Phi(v)}=cz^{\Phi(w)}
\end{eqnarray*}

where $xyz \neq 0, 1$ and $\gcd(x, y, z) = 1$;  $a, b, c, d$ are positive rational integers relatively primes in pairs and $ab \equiv mod 2$; $\Phi$ is a conjugate operator in $\mathbb{Z}[\sqrt{2}]$. Let $u + \Phi(u), v + \Phi(v)$ and $w + \Phi(w)$ be the positive rational integers and $(u - \Phi(u))(v - \Phi(v)) \neq 0$ then the system of equations has no solutions in positive integers. 
\end{proposition}

\begin{proof}
The proof of proposition 3.7 uses the similar concepts with some modifications.
\end{proof}

\section{Discussion and Conclusion}

Fermat, Euler and Legendre saw the Fermat problem as a statement on Diophantine equations and their properties. The Euler observation (1753) about the proofs of Fermat's theorem for the equations $x^3+y^3=z^3$  and $x^4+y^4  = z^4$  permitted to find the real obstacles which have to overcome anyhow solving both the partial cases and the general case of the Fermat equation.

Attracting Legendre's equations (1785) for discussing the Fermat problem (1637) we returned to the Abel interpretation (1823) of Fermat's last theorem as the problem of rational integers of a special form, in our case, on the Legendre curves.

Anyway the Abel approach \cite{Abel:1839aa} as the final result was realized in the proof of Fermat's last theorem as the conflict between the integral solutions and the unique properties of elliptical curves \cite{Frey:aa}. Indeed, it takes place the following statement:

\begin{theorem}
If there exist $u=a^n,v=b^n, w=c^n$, $\gcd(a,b)=\gcd(a,c)=\gcd(b,c)=1$ where $n$ is prime, $n>3$ and $u + v =w$, then an elliptical curve  $\mathbb{E}_{a, b, c}:y^2=x(x-a^n)(x+b^n)$  is not modular.
\end{theorem}

In our note we have discussed the number of cases of applying of Fermat's last theorem and its generalizations proved to solving some problems from algebra, geometry and number theory as well as represented the various points of view on the Fermat problem.

\bibliographystyle{IEEEtran}
\bibliography{references}

\end{document}